\documentclass[12pt]{article}
\usepackage{epsf}
\usepackage{amsbsy,amsmath}
\usepackage{mathtools}
\usepackage{mathrsfs}
\usepackage{amsfonts}
\usepackage{amssymb}
\usepackage{enumitem}
\usepackage{eucal}
\usepackage{enumitem}
\usepackage{graphics,mathrsfs}
\usepackage{amsthm}
\usepackage{secdot}
\usepackage{xcolor}
\newtheorem{theorem}{Theorem}[section]
\newtheorem{lemma}[theorem]{Lemma}

\theoremstyle{definition}
\newtheorem{definition}[theorem]{Definition}

\newtheorem{conjecture}[theorem]{Conjecture}

\theoremstyle{remark}
\newtheorem{remark}[theorem]{Remark}
\numberwithin{equation}{section}

\numberwithin{equation}{section}

\newsavebox{\savepar}

\begin{document}
	
	\title{A study of the {\it Prandtl Batchelor} problem using variational method}
	\author{Debajyoti Choudhuri$^{2}$, Jiabin Zuo$^{1}$\footnote{Corresponding
			author: zuojiabin88@163.com; ORCID ID: 0000-0002-5858-063X}\\
			\small{1 School of Mathematics and Information Science,
Guangzhou University,} \\ \small{Guangzhou 510006, China.}\\
		\small{2 School of Basic Sciences, Indian Institute of Technology Bhubaneswar,}\\ \small{Khordha - 752050, Odisha, India.}}

	\date{}
	\maketitle
	
\begin{abstract}
\noindent In this paper, we investigate the existence of nontrivial weak {solutions} for the {\it Prandtl-Batchelor} type free boundary value elliptic problem driven by a power nonlinearity. The algebraic topology approach will be used to establish the existence of {solutions of approximate problem}, while variational techniques will be used to determine the existence of major problem solutions. In the process several classical results are improved.
\begin{flushleft}
{\bf Keywords}:~ Dirichlet free boundary value problem, Sobolev space, Morse relation, cohomology group.\\
{\bf AMS Classification}:~35J35, 35J60.
\end{flushleft}
\end{abstract}
\section{Introduction}

Consider the following {\it free boundary value} problem
\begin{equation}\label{probmain}
\begin{cases}
-\Delta_{p}u = \lambda\chi_{\{u>1\}} (u-1)_+^{q-1}, &\text{in}~\Omega\setminus H(u)\\
|\nabla u^+|^{p}-|\nabla u^-|^p=\frac{p}{p-1},~&\text{in}~H(u),\\
u = 0, &\text{on}~\partial\Omega
\end{cases}
\end{equation}
where $\lambda> 0$ is a parameter, $(u-1)_+=\max\{u-1,0\}$ and
$$H(u)=\partial\{u>1\}.$$ Also $\nabla u^{\pm}$ are the limits of $\nabla u$ from the sets $\{u>1\}$ and $\{u\leq 1\}^{\circ}$ respectively. The domain $\Omega\subset\mathbb{R}^N (N\geq 2)$ is bounded with a sufficiently smooth boundary $\partial\Omega$ with the {\it exterior sphere condition}. The relation between the exponents are assumed in the order $$2 \leq p \leq q-1, ~\text{with}~ q<p^*=\dfrac{Np}{N-p}.$$  The solution(s) satisfies the free boundary condition in the following sense: {for} all $\vec{\phi}\in C_0^1(\mathbb{R}^N)$ such that $u \neq 1$ a.e. on the support of $\vec{\phi}$,
\begin{align}\label{FBC}
\underset{\epsilon^+\rightarrow 0}{\lim}\int_{u=1+\epsilon^+}\left(\frac{p}{p-1}-|\nabla u|^p\right)\vec{\phi}\cdot\hat{n}dS-\underset{\epsilon^-\rightarrow 0}{\lim}\int_{u=1-\epsilon^-}|\nabla u|^p\vec{\phi}\cdot\hat{n}dS&=0,
\end{align}
where $\hat{n}$ is the outward drawn normal to $\{1-\epsilon^-<u<1+\epsilon^+\}$. Note that the sets $\{u=1\pm\epsilon^{\pm}\}$ are smooth hypersurfaces for almost all $\epsilon^{\pm}>0$ by the Sard theorem. The limit above in \eqref{FBC} is taken by running such $\epsilon^{\pm}>0$ towards zero.

The readers should note that while attempting this problem we have encountered many theorems that are still open to attempt and hence the proposed problem has not been entirely solved for all $p$. A major reason behind this failure is the absence of $C^2$ regularity of solutions to PDEs involving $p$-Laplace operators. However, we have improved the results for the case of $p\geq 2$ wherever it was possible to prove them. For instance, the derivation of the free boundary condition has been done for $p\geq 2$ whereas the proof of the monotonicity lemma in Lemma \ref{conv_res1} has been done from the case of $p=2$. Further, for the interested readers we have conjectured in Conjecture \ref{OP_1} that the monotonicity result may hold if the monotonicity function is chosen in a particular way.

{The readers interested to develope} {a taste of the subject of variational }principles {may refer to the following works} \cite{chou_zamp, ghosh_pos, gu, ren}, \cite{heider2}-\cite{heider1}, \cite{zhangX2}-\cite{zhangX7} {besides the} {classical book by} {\sc Gilbarg-Trudinger} \cite{Gil_trud}.

The trend of applying the topics from algebraic topology in elliptic PDEs is not very old (see \cite{mama}). A rich literature survey has been done in the book due to {\sc Perera et al.} \cite{pererabook} where the {authors have} discussed problems of several variety involving the $p$-Laplace operators which could be studied using the Morse theory. The motivation for the current work has been drawn from the work due to {\sc Perera}  \cite{Perera_NoDEA} {who has considered a sublinear problem}. We also refer the readers to the latest work due to {\sc Choudhuri-Repov\v{s}} \cite{chou_jmaa2}, {{\sc Fotouhi et al.}\cite{Fotouhi_et_al}}. {For a ready reckoner to the techniques used in the study} {of various kinds of elliptic PDEs, we refer the readers to} {\cite{Ferreira_et_al, Sun2_et_al, Sun1_et_al, zhang_et_al}}. The treatment used to address the existence of at least one (or two) solution(s) to the approximating problem may be classical (section $3$, Theorems \ref{thm1} and \ref{twosoln}) but the result concerning the regularity of the free boundary is very new and the question of existence of solution to the problem \eqref{probmain} has not been answered till now (section $4$, Lemma \ref{convergence1}), to the best of my knowledge. A result due to {\sc Alt-Caffarelli} \cite{alt_caffa} (section $4$, Lemma \ref{conv_res1}) were improved to the best possible extent to suit the purpose of the problem in this paper.

{The free boundary value problem} \eqref{probmain} is actually a variant of the well known Prandtl-Batchelor free boundary value problem. Consider the problem

\begin{equation}\label{PB_Prob}
\begin{cases}
-\Delta u = \lambda \chi_{\{u>1\}}(x),~&\text{in}~\Omega\setminus H(u),\\
|\nabla u^+|^{2}-|\nabla u^-|^2=2,~&\text{in}~H(u),\\
u = 0, &\text{on}~\partial\Omega.
\end{cases}
\end{equation}
This is the well known {\it Prandtl-Batchelor} free boundary value problem, where the phase $\{u>1\}$ is a representation of the vortex patch bounded by the vortex line $u=1$ in a steady fluid flow for $N=2$ (refer {\sc Batchelor} \cite{Batchelor1, Batchelor2}). Thus the current problem defined in \eqref{probmain} is a more generalized version of \eqref{PB_Prob}. For a more physical application to this problem we direct the reader's attention to the work due to {\sc Caflisch} \cite{Caflisch}, {\sc Elcrat and Miller} \cite{Elc_mil}.

Another instance of occurrence of such a phenomena is in the non-equilibrium system of melting of ice. In a given block of ice, the heat equation can be solved with a given set of appropriate initial/boundary conditions in order to determine the temperature. However, if there exists a region of ice in which the temperature is greater than the melting point of ice, this subdomain will be filled with water. The boundary thus formed due to the ice-water interface is controlled by the solution of the heat equation.

Thus encountering a {\it free boundary} in the nature is not unnatural. The problem in this paper is a large enough generalization to this physical phenomena which besides being a new addition to the literature can also serve as a note to bridge the problems in elliptic PDEs with algebraic topology.
\section{Preliminaries}
We recall some of the results from algebraic topology that will be used to study the considered problem \eqref{probmain}. Throughout this discourse {we let $X$ be a topological} space and $A\subset X$ be a topological subspace. A fundamental tool that will be used to work with, namely the {\it deformation retraction} (refer Definition $5.3.2$) and the {\it homology theory} (refer Definition $6.1.12$), can be found in the book by {\sc Papageorgiou et al.} \cite{dusanbook}. We now recall the definition of deformation retraction and Homology group on a pair of topological spaces.

\begin{definition}\label{deformation}
	A continuous map $F: X\times [0,1]\to X$ is a deformation retraction of a space $X$ onto a subspace $A$ if, for every $x \in X$ and $a \in A$, $F(x,0)=x$, $F(x,1)\in A$, and $F(a,1)=a$.
\end{definition}
\begin{definition}\label{defn1}
A {\it homology group} on a family of pairs of spaces $(X,A)$ consists of:
\begin{enumerate}
	\item A sequence $\{H_{k}(X,A)\}_{k\in \mathbb{N}_0}$  of abelian groups is known as {\it homology group} for the pair $(X, A)$ (note that for the pair $(X,\phi)$, we write $H_k(X), k \in \mathbb{N}_0$). Here $\mathbb{N}_0=\mathbb{N}\cup\{0\}$.
		\item To every map of pairs $\varphi :(X,A)\rightarrow(Y,B)$ is associated a homomorphism
		$\varphi^{*} : H_k(X,A)\rightarrow H_k(Y,B)$ for all $k \in \mathbb{N}_0$.
			\item 	To every $k \in \mathbb{N}_0$ and every pair $(X,A)$ is associated a homomorphism
			$\partial : H_k(X,A)\rightarrow H_{k-1}(A)$ for all $k \in \mathbb{N}_0$.
		\end{enumerate}
	These items satisfy the following axioms.
	\begin{description}
		\item[($A_1$)] If $\varphi=id_{X}$, then $\varphi_{*}=id|_{H_k(X,A)}$.
		\item[($A_2$)] If $\varphi:(X,A)\rightarrow (Y,B)$ and $\psi:(Y,B)\rightarrow (Z,C)$ are maps of pairs, then $(\psi\circ\varphi)_{*}=\psi_{*}\circ\varphi_{*}$.
		\item[($A_3$)] If $\varphi:(X,A)\rightarrow(Y,B)$ is a map of pairs, then $\partial\circ\varphi_{*}=(\varphi|_{A})_{*}\circ\partial$.
	\item[($A_4$)] If $i:A\rightarrow X$ and $j:(X,\phi)\rightarrow(X,A)$ are inclusion maps, then the following sequence is exact
	$$...\xrightarrow[]{\partial} H_{k}(A)\xrightarrow[]{i_{*}} H_k(X)\xrightarrow[]{j_{*}} H_k(X,A)\xrightarrow[]{\partial} H_{k-1}(A)\rightarrow...$$
		Recall that a chain $...\xrightarrow[]{\partial_{K+1}} C_{k}(X)\xrightarrow[]{\partial_{k}} C_{K-1}(X)\xrightarrow[]{\partial_{k-1}} C_{k-2}(X)\xrightarrow[]{\partial_{k-2}} ...$ is said to be {\it exact} if $im(\partial_{k+1})=ker(\partial_k)$ for each $k\in\mathbb{N}_0$.
	\item[($A_5$)] If $\varphi, \psi:(X,A)\rightarrow(Y,B)$ are homotopic maps of pairs, then $\varphi_{*}=\psi_{*}$.
	\item[($A_6$)] (Excision): If $U\subseteq X$ is an open set with $\bar{U}\subseteq \text{int}(A)$ and $i:(X\setminus U,A\setminus U)\rightarrow (X,A)$ is the inclusion map, then $i_{*}:H_{k}(X\setminus U,A\setminus U)\rightarrow H_k(X,A)$ is an isomorphism.
		\item[($A_7$)] If $X=\{*\}$, then $H_k({*})=0$ for all $k\in\mathbb{N}$.
	\end{description}
\end{definition}

\noindent The deformation lemma which will be quintessential in computing the homology groups can be found in the Lemma $5.5.1$, \cite{kesavanbook}.
We will be using the {\it Palais-Smale} condition which is a special type of compactness that can be referred to the Definition $5.5.1$ in the book \cite{kesavanbook}. Following is the definition of Morse index which will be used in the subsequent sections.
\begin{definition}
	Morse index of a functional $J:V\rightarrow\mathbb{R}$ is defined to be the maximum subspace of $V$ such that $J''$, the second Fr\'{e}chet derivative, is negative definite on it.
\end{definition}
\subsection{Space description}
We begin by defining the standard Lebesgue space $L^{p}(\Omega)$ for $1\leq p<\infty$ as\\
$$L^{p}(\Omega)=\left\lbrace u:\Omega\rightarrow\mathbb{R}:u \;{\text{is measurable and}} \int_{\Omega}|u|^{p}dx < \infty\right\rbrace$$ endowed with the norm
$\|u\|_{p}=\left(\int_{\Omega}|u|^pdx\right)^{\frac{1}{p}}$. We will define the Sobolev space as $$W^{1,p}(\Omega)=\lbrace u\in L^{p}(\Omega):\nabla u\in (L^{p}(\Omega)^N\rbrace$$ with the norm $\|u\|_{1,p}^p=\|u\|_{p}+\|\nabla u\|_{p}$. We further define
$$W_0^{1,p}(\Omega)=\lbrace u\in W^{1,p}(\Omega):u=0~\text{on}~\partial\Omega\rbrace.$$
The associated norm to the space will be denoted by
$\|u\|=\|\nabla u\|_{p}$. With these norms, $L^{p}(\Omega)$, $W^{1,p}(\Omega)$ and $W_0^{1,p}(\Omega)$ are separable, reflexive Banach spaces(\cite{kesavanbook}).
\noindent The embedding results pertaining to the Sobolev spaces can be referred to the Lemma $2.4.1$ in \cite{kesavanbook}.
\section{The way to tackle the problem using Morse theory}
We at first define an energy functional associated to the problem in \eqref{probmain} which is as follows.
\begin{align*}
\begin{split}
I(u)&=\int_{\Omega}\frac{|\nabla u|^{p}}{p}dx+\int_{\Omega}\chi_{\{u>1\}}(x)dx-\lambda\int_{\Omega}\frac{(u-1)_+^{q}}{q}dx.
\end{split}
\end{align*}
However, this functional is not even differentiable and hence poses serious issues as far as the application of variational theorems are concerned. Thus we approximate $I$ using the following functionals that varies with respect to a parameter $\alpha>0$. This method is adapted from the work of {\sc Jerison-Perera} \cite{J_P_1}. We define a smooth function $g:\mathbb{R}\rightarrow [0,2]$ as follows:
\[g(t)= \begin{cases}\label{smooth_func_1}
0, & \text{if}~t\leq 0 \\
g(t)>0, & \text{if}~0<t<1\\
0, & \text{if}~t\geq 1
\end{cases}\]
and $\int_0^1g(t)dt=1$.
We further let $G(t)=\int_0^tg(t)dt$. Clearly, $G$ is smooth and nondecreasing function such that
\[G(t)= \begin{cases}\label{smooth_func_2}
0, & \text{if}~t\leq 0 \\
0<G(t)<1, & \text{if}~0<t<1\\
1, & \text{if}~t\geq 1.
\end{cases}\]
We thus define
\begin{align*}
\begin{split}
I_{\alpha}(u)&=\int_{\Omega}\frac{|\nabla u|^{p}}{p}dx+\int_{\Omega}G\left(\frac{u-1}{\alpha}\right)dx-\lambda\int_{\Omega}\frac{(u-1)_+^{q}}{q}dx.
\end{split}
\end{align*}
This functional $I_{\alpha}$, is of at least $C^2$ class and hence
\begin{align*}
\langle I_{\alpha}''(u)v,w\rangle=&\int_{\Omega}[|\nabla u|^{p-2}\nabla v\cdot\nabla w+(p-2)|\nabla u|^{p-4}(\nabla u\cdot\nabla v)(\nabla u\cdot\nabla w)]dx\\
&+\int_{\Omega}\frac{1}{\alpha^2}g'\left(\frac{u-1}{\alpha}\right)vwdx-\lambda\int_{\Omega}(u-1)_+^{q-2}vwdx.
\end{align*}
Following is an important result in Morse theory which explains the effect of the associated Homology groups on the set $K_{J,(-\infty,a]}=\{x\in V:J(x)\leq a\}$.
\begin{theorem}\label{homology1}
Let $J\in C^2(V,\mathbb{R})$ satisfy the Palais-Smale condition and let `$a$' be a regular value of $J$. Then if, $H_*(V,J^a)\neq 0$, implies that $K_{J,(-\infty,a]}\neq\emptyset$.
\end{theorem}
\begin{remark}\label{rem1}
Before we apply the Morse lemma we recall that for a Morse function the following holds
\begin{enumerate}
 \item $$H_{*}(J^c,J^c\setminus \text{Crit}(J,c))=\oplus_{j}H_*(J^c\cap N_j,J^c\cap N_j\setminus\{x_j\}),$$
 where $\text{Crit}(J,c)=\{x\in V:J(x)=c, J'(x)=0\}$, $N_j$ is a neighbourhood of $x_j$.
 \item \[H_{k}(J^c\cap N,J^c\cap N\setminus\{x\})= \begin{cases}
\mathbb{R}, & k=m(x) \\
0, & \text{otherwise}
\end{cases}\]
where $m(x)$ is a Morse index of $x$, a critical point of $J$.\\
\item Further $$H_k(J^a,J^b)=\oplus_{\{i:m(x_i)=k\}}\mathbb{R}=\mathbb{R}^{m_k(a,b)}$$ where $m_k(a,b)=n(\{i:m(x_i)=k,x_i\in K_{J,(a,b)}\})$. Here $n(S)$ denotes the number of elements present in the set $S$.
\item Morse relation
$$\sum_{u\in K_{J,[a,b]}}\sum_{k\geq 0}\text{dim}(C_k(J,u)) t^k=\sum_{k\geq 0}\text{dim}(H_k(J^{a},J^{b})) t^k+(1+t)\mathcal{Q}_t$$ for all $t\in\mathbb{R}$. Here $\mathcal{Q}_t$ is a nonnegative polynomial in $\mathbb{N}_0[t]$.
\end{enumerate}
\end{remark}
\begin{theorem}\label{thm1}
The functional $I_{\alpha}$ has at least one nontrivial critical point when $0<\lambda\leq \lambda_1$, $\lambda_1$ being the first eigenvalue of $(-\Delta_{p})$.
\end{theorem}
\begin{proof}
We observe that $I_{\alpha}(tu)\rightarrow-\infty$ as $t\rightarrow\infty$. Furthermore, a key observation here is that there exists $(r,\lambda)$ sufficiently small positive numbers such that
$$I_{\alpha}(u)\geq A>0$$ whenever $\|u\|=r$. This demonstrates that the functional simply has the Mountain Pass geometry.

We note that $I_{\alpha}$ obeys the Palais-Smale $(PS)$ condition, for $p\geq 2$. In order to prove this, we define $$u_n^+(x):=\max\{u_n(x),0\}, u^++u^-:=(u-1)_++[1-(u-1)_-]=u.$$
	We further note that
	\begin{align}\label{PS_cond_pf}
	\begin{split}
	I_{\alpha}(u_n)&\geq p^{-1}\|u_n\|^p-\frac{\lambda}{q}\int_{\Omega}(u_n)_+^qdx\\
	\langle	I_{\alpha}'(u_n),u_n\rangle&\leq \|u_n\|^p-\lambda\int_{\Omega}(u_n^+)^qdx+\frac{2}{\alpha}|\Omega|.
	\end{split}
	\end{align}
	Consider $c\in\mathbb{R}$ and
	\begin{align}\label{PS_cond_pf1}
	\begin{split}
	c+\sigma\|u_n\|+o(1)\geq I_{\alpha}(u_n)-\frac{1}{q}\langle	I_{\alpha}'(u_n),u_n\rangle&\geq \left(p^{-1}-q^{-1}\right)\|u_n\|^p-\frac{2}{\alpha}|\Omega|.
	\end{split}
	\end{align}
	This implies that $(u_n)$ is bounded in $W_0^{1,p}(\Omega)$. This implies that there exists a subsequence of $(u_n)$ such that $u_n\rightharpoonup u$ in $W_0^{1,p}(\Omega)$, $u_n\to u$ in $L^q(\Omega)$ and $u_n(x)\to u(x)$ a.e. in $\Omega$. Since $\langle	I_{\alpha}'(u_n),v\rangle\to 0$ as $n\to\infty$ we have
	\begin{align}\label{PS_cond_pf2}
	\begin{split}
	\underset{n\to\infty}{\lim}\int_{\Omega}|\nabla u_n|^{p-2}{\nabla}u_n\cdot\nabla vdx&=	 \underset{n\to\infty}{\lim}\left[\int_{\Omega}\frac{1}{\alpha}g\left(\frac{u_n-1}{\alpha}\right)vdx+\lambda\int_{\Omega}(u_n-1)_+^{q-1}vdx\right]
	\end{split}
	\end{align}
for all $v\in W_0^{1,p}(\Omega)$. In particular on choosing $v=u_n-u$ in \eqref{PS_cond_pf2} we obtain
	\begin{align}\label{PS_cond_pf3}
	\begin{split}
	\underset{n\to\infty}{\lim}\int_{\Omega}|\nabla u_n|^{p-2}{\nabla}u_n\cdot\nabla (u_n-u)dx=	 \underset{n\to\infty}{\lim}&\left[\int_{\Omega}\frac{1}{\alpha}g\left(\frac{u_n-1}{\alpha}\right)(u_n-u)dx\right.\\
	&\left.+\lambda\int_{\Omega}(u_n-1)_+^{q-1}(u_n-u)dx\right]=0.
	\end{split}
	\end{align}
	Hence $u_n\to u$ in $W_0^{1,p}(\Omega)$. Therefore the functional $I_{\alpha}$ satisfies the (PS) condition.
We choose $\epsilon>0$ such that $c=\epsilon$ is a regular value of $I_{\alpha}$. Thus, $I_{\alpha}^{\epsilon}$ is not path connected since it has at least two path connected components namely in the form of a neighbourhood of $0$ and a set $\{u:\|u\|\geq R\}$ for $R$ sufficiently large. Therefore, from the theory of homology groups we get that $\text{dim} (H_0(I_{\alpha}^{\epsilon}))\geq 2$; `dim' denoting the dimension of the Homology group. From the Definition \ref{defn1}-$(A_4)$ of homology group, let us consider the following exact sequence
$$...\rightarrow H_1(W_{0}^{1,p}(\Omega),I_{\alpha}^{\epsilon})\xrightarrow[]{\partial_1}H_0(I_{\alpha}^{\epsilon},\emptyset)\xrightarrow[]{i_0}H_0(W_{0}^{1,p}(\Omega),\emptyset)\rightarrow...$$

We know that $\text{dim} (H_0(W_{0}^{1,p}(\Omega),\emptyset))=1$ and $\text{dim}(H_0(I_{\alpha}^{\epsilon}))\geq 2$. Moreover, due to the exactness of the sequence we conclude that $\text{dim}H_1(W_{0}^{1,p}(\Omega),I_{\alpha}^{\epsilon})\geq 1$. Thus by the Theorem \ref{homology1} we have $K_{I_{\alpha},(-\infty,\epsilon]}\neq\emptyset$.

Suppose that the only critical point to \eqref{probmain} is $u=0$ at which the energy of the functional $I_{\alpha}$ is also $0$. Thus from the discussion above and the Remark \eqref{rem1}-(4) we have from the Morse relation the following identity over $\mathbb{R}$
$$1=t+\mathcal{P}(t)+(1+t)\mathcal{Q}_t,$$ $\mathcal{P}$ being a power series in any $t\in\mathbb{R}$, $\mathcal{Q}_t\geq 0$. This leads to a contradiction since a nonzero polynomial can never be an identity on $\mathbb{R}$. Thus there exists at least one $u\neq 0$ which is a critical point to $I_{\alpha}$ whenever $\lambda\leq \lambda_1$.
\end{proof}
\noindent
\begin{remark}\label{moser}
Without any demonstration we quote that a standard application of the Moser iteration technique leads to the conclusion that a solution $u$ of \eqref{probmain} is in $L^{\infty}(\Omega)$.
\end{remark}
\noindent We refer the reader to the Definition $1.3$ in \cite{ghosh_pos} for the definition of {\it Krasnoselskii genus}.
 \noindent To each closed and symmetric subsets $M$ of $W_0^{1,p}(\Omega)$ with the Krasnoselskii genus $\gamma(M)\geq k$, define
 $$\lambda_k=\inf_{M\in\mathfrak{F}_k}\sup_{u\in M}I_{\alpha}(u).$$  Here $\mathfrak{F}_k=\{M\subset W_0^{1,p}(\Omega),~\text{closed and symmetric}:\gamma(M)\geq k\}$. A natural question at this point will be to ask if the same conclusion as in Theorem \ref{thm1} can be drawn when $\lambda_k<\lambda\leq\lambda_{k+1}$. We will define $\lambda_0=0$. The next theorem answers this question.
\begin{theorem}\label{twosoln}
	The problem in \eqref{probmain} has at least one nontrivial solution when $\lambda_i<\lambda\leq\lambda_{j}$, $\lambda_i, \lambda_j$ being as defined above, for some $i,j$.
\end{theorem}
\begin{proof}
We at first show that $H_k(W_0^{1,p}(\Omega),I_{\alpha}^{-a})=0$ for all $k\geq 0$.	Towards this we pick a $u\in\{v:\|v\|=1\}=\partial B^{\infty}$, where $B^{\infty}=\{v:\|v\|\leq 1\}$. Therefore there exists $t_0>0$ such that $$I_{\alpha}(tu)=\int_{\Omega}\frac{|\nabla (tu)|^{p}}{p}dx+\int_{\Omega}G\left(\frac{tu-1}{\alpha}\right)dx-\lambda\int_{\Omega}\frac{(tu-1)_+^{q}}{q}dx< -a<0$$ for all $t\geq t_0$. We observe from the Remark \ref{moser} that for $t>0$ small enough, the sign of $tu-1$ becomes negative, whence, there exists $\bar{t}>0$ such that $I_{\alpha}'(\bar{t}u)>0$. Thus, there exists $t(u)$ such that $I_{\alpha}'(t(u)u)=0$ by the continuity of $I_{\alpha}'$. We can thus say that there exists a $C^1$-function $T:W_0^{1,p}(\Omega)\setminus\{0\}\rightarrow\mathbb{R}^{+}$. We now define a standard  deformation retract $\eta$ of $W_0^{1,p}(\Omega)\setminus B_{R'}(0)$ into $I_{\alpha}^{-a}$ as follows:
 \[\eta(s,u)=\begin{cases}
(1-s)u+sT\left(\frac{u}{\|u\|}\right)\frac{u}{\|u\|}, & \|u\|\geq R', I_{\alpha}(u)\geq -a\\
u, & I_{\alpha}(u)\leq -a.
\end{cases}\]
It is not difficult to see that $\eta$ is a $C^1$ function over $[0,1]\times W_0^{1,p}(\Omega)\setminus B_{R'}(0)$. On using the map $\Theta(s,u)=\dfrac{u}{\|u\|}$, for $u\in W_0^{1,p}(\Omega)\setminus B_{R'}(0)$ we claim that $H_k(W_0^{1,p}(\Omega),W_0^{1,p}(\Omega)\setminus B_r(0))=H_k(B^{\infty},S^{\infty})$ for all $k\geq 0$. This is because, $H_k(B^{\infty},S^{\infty})\cong H_k(*,0)$. From the elementary computation of homology groups with two {\it $0$-dimensional simplices} it is easy to see that $H_k(*,0)=\{0\}$ for each $k\geq 0$.
Therefore, from the Morse relation in the Remark \eqref{rem1}-4 and the result above, we have for $b>0$
\begin{align}
\sum_{u\in K_{I,[-a,\infty)}}\sum_{k\geq 0}\text{dim}(C_k(I,u))t^k&=t^{m(u)}+p(t)
\end{align}
where $m(u)$ is the Morse index of $u$ and $\mathcal{P}(t)$ contains the rest of the powers of $t$ corresponding to the other critical points, if any. The Morse index is finite because of the following reason:~ The mountain pass geometry around $0$ allows in establishing a maxima (say $u_0$) along with the assumption $\lambda<C^{-q}\frac{q}{p}\|u\|^{p-q}$. Owing to $u_0$ being a maxima, we have $I_{\alpha}''(u_0)<0$ which necessarily requires $\lambda>C^{-q}\frac{p-1}{q-1}\|u\|^{p-q}$. Thus we have $$C^{-q}\frac{p-1}{q-1}\|u\|^{p-q}<\lambda<C^{-q}\frac{q}{p}\|u\|^{p-q}.$$ This implies that $\lambda_i<\lambda<\lambda_j$ for some $i,j\in\mathbb{N}_0$. On further using the Morse relation we obtain
 \begin{align}
 t^{m(u)}+\mathcal{P}(t)&=(1+t)\mathcal{Q}_t.
 \end{align}
 This is because the $H_k$s are all trivial groups. Hence, $Q_t$ either contains $t^{m(u)}$ or $t^{m(u)-1}$ or both. Thus there exists at least one nontrivial $u\in K_{I_{\alpha},[-a,\infty)}$ with $m(u)<\infty$.
\end{proof}
\begin{remark}\label{k_solns}
As a note if $0<\lambda\leq\lambda_{k+1}$, then there exists at least $k$ solutions to the equation \eqref{probmain} when considered over domains with sufficiently regular boundary.
\end{remark}
\begin{remark}[see Iannizzotto et al. \cite{Ian2}]\label{p_lap}
Suppose $D$ is a bounded domain with a sufficiently smooth boundary, then the problem
\begin{align}\label{prob_p_laplace}
	\begin{split}
		-\Delta_pu=&f~\text{in}~D
	\end{split}
\end{align}
has a $C_{\text{loc}}^{1,a}$ solution whenever $f$ is bounded and nothing more can be expected, irrespective of the smoothness of $f$.
\end{remark}
\section{Existence of solution to the main problem \eqref{probmain} and smoothness of the boundary $\partial\{u>1\}$}
The following lemma has been proved for $p=2$.
\begin{lemma}\label{convergence1}
	Let $\alpha_j\rightarrow 0$ ($\alpha_j>0$) as $j\rightarrow\infty$ and $u_j$ be a critical point of $I_{\alpha_j}$. If $(u_j)$ is bounded in $W_0^{1,2}(\Omega)\cap L^{\infty}(\Omega)$, then there exists $u$, a Lipschitz continuous function, on $\bar{\Omega}$ such that $u\in W_0^{1,2}(\Omega)\cap C^{1,a}(\bar{\Omega}\setminus H(u))$ and a subsequence (still denoted by $(u_j)$) such that
	\begin{enumerate}[label=(\roman*)]
		\item $u_j\rightarrow u$ uniformly over $\bar{\Omega}$,
		\item $u_j\rightarrow u$ locally in $C^1(\bar{\Omega}\setminus\{u=1\})$,
		\item $u_j\rightarrow u$ strongly in $W_0^{1,2}(\Omega)$,
		\item $I(u)\leq\lim\inf I_{\alpha_j}(u_j)\leq\lim\sup I_{\alpha_j}(u_j)\leq I(u)+|\{u=1\}|$, i.e. $u$ is a nontrivial function if $\lim\inf I_{\alpha_j}(u_j)<0$ or $\lim\sup I_{\alpha_j}(u_j)>0$.
	\end{enumerate}
	Furthermore, $u$ is a $C^{1,a}$-weak solution of $$-\Delta u=\lambda\chi_{\{u>1\}}(x)(u-1)_+^{q-1}$$ in $\Omega\setminus H(u)$, the free boundary condition is satisfied in the generalized sense and vanishes continuously on $\partial\Omega$. In the case of $u$ being nontrivial, then $u>0$ in $\Omega$, the set $\{u<1\}$ is connected and the set $\{u>1\}$ is nonempty.
\end{lemma}
\noindent An important result that will be used to pass the limit in the proof of the Lemma \ref{convergence1} is the following theorem which is in line to the theorem due to {\sc Caffarelli et al.} in \cite[Theorem $5.1$]{Caffa_jeri_kenig}.
\begin{lemma}\label{conv_res1}
	Let $u$ be a Lipschitz continuous function on the unit ball $B_1(0)\subset\mathbb{R}^N$ satisfying the distributional inequalities
	\begin{align}\label{dist_equ}
	\pm\Delta u&\leq \left(\dfrac{1}{\alpha}\chi_{\{|u-1|<\alpha\}}(x)\mathcal{H}(|\nabla u|^2)+A\right)
	\end{align}
	for constants $A>0$, $0<\alpha\leq 1$ and $\mathcal{H}$ is a continuous function such that $\mathcal{H}(t)=o(t^2)$ near infinity. Then there exists a constant $C>0$ depending on $N, A$ and $B=\int_{{B_1}(0)}u^2dx$, but not on $\alpha$, such that
	$$\underset{x\in B_{\frac{1}{2}}(0)}{\text{max}}\{|\nabla u(x)|\}\leq C.$$
\end{lemma}
\begin{proof}
The proof of $(i)-(iv)$ follows from \cite{Perera_NoDEA}. We now prove the following for $p\geq 2$:

There exists constant $C_1$ such that for any $\epsilon>0$ there exists a finite and positive number $\mu_0$ such that for any $\mu\geq\mu_0$ then
\begin{enumerate}
	\item $|S_{\mu}\cap Q_0|\leq |S_{\mu_0}\cap Q_0|<\epsilon|Q_0|$. Here $Q_0$ is a cube with length of each side being $2^{-10-10N}$ and $Q_0\cap B_{1/32}\neq\emptyset$.
	\item Suppose $Q$ is a dyadic subcube of $Q_0$ for which $|S_{C_1\mu}\cap Q|\geq \epsilon|Q|$, then $Q\subset Q^*\subset S_{\mu}$, where $Q$ is an immediate dyadic subcube of $Q^*$.
\end{enumerate}
{\it Proof}:~We only sketch the proof as the ideas are borrowed from \cite{Caffa_jeri_kenig}. The claim in $1.$ follows from the argument given in \cite{Caffa_jeri_kenig}.\\
To begin with, the maximal operator is defined as follows:
\begin{align}\label{max_oper}
\mathfrak{M}f(x)&=\underset{0<r<1/100}{\sup}\frac{1}{|B_r(x)|}\int_{B_r(x)}f(y)dy.
\end{align}
We will further denote $$S_{\mu}=\{x\in B_{1/32}:\mathfrak{M}(|\nabla v_1|^p)(x)>\mu\}.$$

Suppose $2.$ fails to hold. Then one can find a cube $Q$ such that $|S_{\mu}\cap Q|\geq\epsilon|Q|$ and $y\in Q^*$ but $\mathfrak{M}(|\nabla v_1|^p)(y)\leq \mu$. Let $\rho$ be $2^{6N}$ times the length of the sides of $Q$ and consider $\mathfrak{M}_{\rho/4}(|\nabla v_1|^p)(0)$, with the supremum taken over $(0,\rho/4)$. Since $\mathfrak{M}(|\nabla v_1|^p)(y)\leq \mu$, there exists a constant $C_2$ such that for any $x\in Q$,
\begin{align}\label{est1}
\mathfrak{M}(|\nabla v_1|^p)(x)&\leq\max\{\mathfrak{M}_{\rho/4}(|\nabla v_1|^p)(x),C_2\mu\}.
\end{align}
Let $\phi$ be such that
\begin{align}\label{harm_func}
\begin{split}
-\Delta_p\phi&=0~\text{in}~B_{\rho}(y)\\
\phi&=v_1~\text{on}~\partial B_{\rho}(y).
\end{split}
\end{align}
Since $\phi$ is a minimizer of the functional $\frac{1}{2}\int_{B_{\rho}(y)}|\nabla\phi|^p dx$, we have
\begin{align}\label{est2}
\int_{B_{\rho}(y)}|\nabla\phi|^p dx&\leq \int_{B_{\rho}(y)}|\nabla v_1|^p dx\leq\mu|B_{\rho}(y)|.
\end{align}
Of course, we have the mean value theorem (refer Theorem $1$ in {\sc Lindqvist-Manfredi} \cite{LIND}) at our disposal to guarantee that
\begin{align}\label{est3}
\underset{B_{\rho/2}(y)}{\sup}\{|\nabla \phi|^p\}&\leq C_3\mu.
\end{align}
On choosing $C_1=15\max\{C_2,C_3\mu\}$ we have
\begin{align}\label{est4}
\mathcal{A}:=\{x\in Q:\mathfrak{M}_{\rho/4}(|\nabla v_1|^p)(x)>C_1\mu\}=\{x\in Q:\mathfrak{M}(|\nabla v_1|^p)(x)>C_1\mu\}=:\mathcal{B}.
\end{align}
If $x\in\mathcal{A}$, then it is easy see that $x\in\mathcal{B}$. Thus $\mathcal{A}\subset\mathcal{B}$. Suppose $x\in\mathcal{B}$, then $\mathfrak{M}(|\nabla v_1|^p)(x)>C_1\mu$. However, by \eqref{est1} and by the choice of $C_1$ we have that $x\in\mathcal{A}$.\\
\noindent Also observe that $\{x\in Q:\mathfrak{M}_{\rho/4}(|\nabla \phi|^p)>C_1\mu/4\}=\emptyset$. For if not, then there exists $x\in Q$ such that $\mathfrak{M}(|\nabla \phi|^p)(x)>C_\mu/4$. One may thus produce $r\in (0,\rho/4)$ such that $$\frac{C_1\mu}{4}<\frac{1}{|B_r(x)|}\int_{B_r(x)}|\nabla \phi|^pdy\leq\frac{C_1\mu}{15}.$$
This is a contradiction since this leads to an absurdity $4>15$.\\
Therefore,
\begin{align}\label{est5}
\begin{split}
|\{x\in Q&: \mathfrak{M}_{\rho/4}(|\nabla v_1|^p)>C_1\mu\}|\\
&\leq|\{x\in Q: \mathfrak{M}_{\rho/4}(|\nabla (v_1-\phi)|^p)+\mathfrak{M}_{\rho/4}(|\nabla\phi|^p)>C_1\mu/2\}|\\
&\leq|\{x\in Q: \mathfrak{M}_{\rho/4}(|\nabla (v_1-\phi)|^p)>C_1\mu/4\}|+|\{\mathfrak{M}_{\rho/4}(|\nabla\phi|^p)>C_1\mu/4\}|\\
&=|\{x\in Q: \mathfrak{M}_{\rho/4}(|\nabla (v_1-\phi)|^p)>C_1\mu/4\}|.
\end{split}
\end{align}
Thus there exists a constant $C_4$, which follows by the weak $(1,1)$ inequality of $\mathfrak{M}$, such that
\begin{align}\label{est6}
C_4\mu^{-1}\int_{B_{\rho}(y)}|\nabla(v_1-\phi)|^pdx&\geq|\{x\in Q:\mathfrak{M}_{\rho/4}(|\nabla(v_1-\phi)|^p)>C_1\mu/4\}|.
\end{align}
Furthermore, by the maximum principle we have $|v_1-\phi|\leq C$ on the ball $B_{\rho}(y)$. By the weak formulation of the problem \eqref{harm_func}, we have
\begin{align}\label{est7}
0&=\int_{B_{\rho}(y)}|\nabla \phi|^{p-2}\nabla \phi\cdot\nabla(v_1-\phi)dx
\end{align}
and
\begin{align}\label{est8}
\begin{split}
-\int_{B_{\rho}(y)}\Delta_pv_1(v_1-\phi)dx&=-\int_{B_{\rho}(y)}(\Delta_pv_1-\Delta_p\phi)(v_1-\phi)dx.
\end{split}
\end{align}
Thus by the Simon  inequality (see (4.6) \cite{chou_zamp}) we have
\begin{align}\label{est8'}
\begin{split}
C_5\int_{B_{\rho}(y)}|\nabla(v_1-\phi)|^pdx&\leq \int_{B_{\rho}(y)}(|\nabla v_1|^{p-2}\nabla v_1-|\nabla \phi|^{p-2}\nabla \phi)\cdot\nabla(v_1-\phi)dx\\
&\leq -\int_{B_{\rho}(y)}(\Delta_pv_1-\Delta_p\phi)(v_1-\phi)dx\\
&=-\int_{B_{\rho}(y)}(\Delta_pv_1)(v_1-\phi)dx\\
&\leq\int_{B_{\rho}(y)}C\left(\beta|v_1|^p+A(\beta)\right)dx.
\end{split}
\end{align}
Using the above inequality \eqref{est6} obtained, we get
\begin{align}\label{est9}
\begin{split}
|\{x\in Q:\mathfrak{M}_{\rho/4}(|\nabla v_1|^p)>C_1\mu\}|&\leq C_6\left(\beta+\frac{A(\beta)}{\mu}\right)|Q|.
\end{split}
\end{align}
Thus, for a sufficiently small $\delta>0$ and large $\mu>0$ we have $C_6\delta<\epsilon/3$ and $C_6A(\beta)/\mu<\epsilon/3$. Therefore
$$\{x\in Q: \mathfrak{M}(|\nabla v_1|^p)>C_1\mu\}<\epsilon|Q|$$
which indeed is a contradiction to the hypothesis. Therefore, $2.$ holds.\\
One may now follow verbatim of \cite{Caffa_jeri_kenig} to conclude that the conclusion of $2.$ leads to
\begin{align}\label{est10}
\begin{split}
|S_{C_1^k\mu}\cap Q_0|&\leq\epsilon^{k+1}|Q_0|.
\end{split}
\end{align}

We now note from \eqref{est10} that for any $1<\theta<\infty$ a sufficiently small $\epsilon>0$ can be chosen so that $\mathfrak{M}(|\nabla v_1|^p)$ is bounded in $L^{\theta}(B_{1/16})$, i.e.
\begin{align}\label{est11}
\begin{split}
\int_{B_{1/16}}|\nabla v_1|^{\theta}dx&\leq C_{7},
\end{split}
\end{align}
where $C_7$ is a uniform constant that depends on $\theta, A, \mathcal{H}$. On choosing $\theta=N$ we have $2\theta>N$. Hence from the condition on $\mathcal{H}$, we have that for every $\beta\in(0,1]$, there exists a positive number $A(\beta)$ such that $\mathcal{H}(t)\le A(\beta)+\beta t^2$. Thus  we obtain
\begin{align}\label{est12}
\begin{split}
\underset{B_{1/32}}\sup\{|\nabla v_1|\}&\leq C_8.
\end{split}
\end{align}
Reverting back to the variables in terms of $u$ we get
\begin{align}\label{est13}
\begin{split}
\underset{B_{\alpha/320}(x)}\sup\{|\nabla u|\}&\leq C_8~\text{for any}~x\in B_{1/4}~\text{such that}~|\{u(x)<\alpha\}|.
\end{split}
\end{align}
To finally arrive at the conclusion
\begin{align}\label{est14}
\begin{split}
\underset{B_{r/4}(0)}\sup\{|\nabla u|\}&\leq C_9.
\end{split}
\end{align}
we follow the proof of \cite{Caffa_jeri_kenig} again, however with the choice of $w(x)=A_0r(r^{N-p}|x|^{p-N}-1)+A(|x|^p-r^p)+O(\alpha)$. Therefore $\underset{B_{r/2}}{\sup}\{|\nabla u|\}<\infty$.
\end{proof}

\begin{remark}\label{H-cond}
The condition $\mathcal{H}(t)=o(t^2)$ in Theorem \ref{conv_res1} can also be replaced by $\mathcal{H}(t)=o(t^p)$.
\end{remark}

\begin{proof}[Proof of Lemma \ref{convergence1}]
	Let $0<\alpha_j<1$. Consider the problem sequence $(P_j)$
	\begin{align}\label{approx_prob_1}
	\begin{split}
	-\Delta u_j&=-\frac{1}{\alpha_j}g\left(\frac{(u_j-1)_+}{\alpha_j}\right)+\lambda (u-1)_+^{q-1}~\text{in}~\Omega\\
	u_j&>0~\text{in}~\Omega\\
	u_j&=0~\text{on}~\partial\Omega.
	\end{split}
	\end{align}
	The nature of the problem being sublinear allows us to conclude by an iterative technique that the sequence $(u_j)$ is bounded in $L^{\infty}(\Omega)$. Therefore, there exists $C_0$ such that $0\leq g\left(\frac{(u_j-1)_+}{\alpha_j}\right)(u-1)_+^{q-1}\leq C_0$. Let $\varphi_0$ be a solution of
	\begin{align}\label{approx_prob_2}
	\begin{split}
	-\Delta \varphi_0&=\lambda C_0~\text{in}~\Omega\\
	\varphi_0&=0~\text{on}~\partial\Omega.
	\end{split}
	\end{align}
	Now since $g\geq 0$, we have that $-\Delta u_j\leq \lambda C_0=-\Delta\varphi_0$ in $\Omega$. Therefore by the maximum principle,
	\begin{align}\label{comparison1}
	0\leq u_j(x)\leq \varphi_0(x)~\forall x\in\Omega.
	\end{align}
	Since $\{u_j\geq 1\}\subset\{\varphi_0\geq 1\}$, hence $\varphi_0$ gives a uniform lower bound, say $d_0$, on the distance from the set $\{u_j\geq 1\}$ to $\partial\Omega$. Thus $(u_j)$ is bounded with respect to the {$C^{1,a}$-norm}. Therefore, it has a convergent subsequence in the {$C^{1,a}$-norm} in a $\dfrac{d_0}{2}$ neighbourhood of the boundary $\partial\Omega$. Obviously $0\leq g\leq 2\chi_{(-1,1)}$ and hence
	\begin{align}\label{comparison2}
	\begin{split}
	\pm\Delta u_j&=\pm\frac{1}{\alpha_j}g\left(\frac{(u_j-1)_+}{\alpha_j}\right)\mp\lambda (u_j-1)_+^{q-1}\\
	&\leq\frac{2}{\alpha_j}\chi_{\{|u_j-1|<\alpha_j\}}(x)+\lambda C_0.
	\end{split}
	\end{align}
	Since, $(u_j)$ is bounded in $L^2(\Omega)$ and by Lemma \ref{conv_res1} it follows that there exists $A>0$ such that
	\begin{align}\label{aux1}
	\underset{x\in B_{\frac{r}{2}}(x_0)}{\text{esssup}}\{|\nabla u_j(x)|\}&\leq\frac{A}{r}
	\end{align}
	for a suitable $r>0$ such that $B_r(0)\subset\Omega$. However, since $(u_j)$ is a sequence of Lipschitz continuous functions that are also $C^1$, therefore
	\begin{align}\label{aux2}
	\underset{x\in B_{\frac{r}{2}}(x_0)}{\sup}\{|\nabla u_j(x)|\}&\leq\frac{A}{r}.
	\end{align}
	Thus $(u_j)$ is uniformly Lipschitz continuous on the compact subsets of $\Omega$ such that its distance from the boundary $\partial\Omega$ is at least $\frac{d_0}{2}$ units.

	Thus by the {\it Ascoli-Arzela} theorem applied to $(u_j)$ we have a subsequence, still named the same, such that it converges uniformly to a Lipschitz continuous function $u$ in $\Omega$ with zero boundary values and with a $C^1$ convergence on a $\frac{d_0}{2}$-neighbourhood of $\partial\Omega$. By the {\it Eberlein-\v{S}mulian} theorem we conclude that $u_j\rightharpoonup u$ in $W_0^{1,2}(\Omega)$.\\
	We now prove that $u$ satisfies
	\begin{align}\label{aux_prob_sat}-\Delta u&=\alpha\chi_{\{u>1\}}(x)(u-1)_+^{q-1}\end{align}
	in the set $\{u\neq 1\}$. Let $\varphi\in C_0^{\infty}(\{u>1\})$ and therefore $u\geq 1+2\delta$ on the support of $\varphi$ for some $\delta>0$. On using the convergence of $u_j$ to $u$ uniformly on $\Omega$ we have $|u_j-u|<\delta$ for any sufficiently large $j,\delta_j<\delta$. So $u_j\geq 1+\delta_j$ on the support of $\varphi$. On testing \eqref{aux_prob_sat} with $\varphi$ yields
	\begin{align}\label{weak_conv_1}
	\int_{\Omega}\nabla u_j\cdot\nabla\varphi dx&=\lambda\int_{\Omega}(u_j-1)_+^{q-1}\varphi dx.
	\end{align}
	On passing the limit $j\rightarrow\infty$ to \eqref{aux_prob_sat}, we get
	\begin{align}\label{weak_conv_2}
	\int_{\Omega}\nabla u\cdot\nabla\varphi dx&=\lambda\int_{\Omega}(u-1)_+^{q-1}\varphi dx.
	\end{align}
	To arrive at \eqref{weak_conv_2} we have used the weak convergence of $u_j$ to $u$ in $W_0^{1,2}(\Omega)$ and the uniform convergence of the same in $\Omega$. Hence $u$ is a weak solution of $-\Delta u=\lambda (u-1)_+^{q-1}$ in $\{u>1\}$. Since $u$ is a Lipschitz continuous function, hence it is also a solution of $-\Delta u=\lambda (u-1)_+^{q-1}$ in $\{u>1\}$. Similarly on choosing $\varphi\in C_0^{\infty}(\{u<1\})$ one can find a $\delta>0$ such that $u\leq 1-2\delta$. Therefore, $u<1-\delta$.\\
On testing \eqref{aux_prob_sat} with any nonnegative function and passing the limit $j\rightarrow\infty$ and using the fact that $g\geq 0$, $G\leq 1$ we can show that $u$ satisfies
	\begin{align}\label{weak_conv_3}
	-\Delta u&\leq\lambda (u-1)_+^{q-1}~\text{in}~\Omega
	\end{align}
	in the distributional sense. We note here that the set $\{u<1\}$ is of nonzero Lebesgue measure since $u$ is continuous.
	Furthermore, we claim that $\Delta (u-1)_-\geq 0$ in the distributional sense. {The proof is as follows:}

	We follow the proof due to {\sc Alt-Caffarelli} \cite{alt_caffa}. Choose $\delta>0$ and a test function $\varphi^2\chi_{\{u<1-\delta\}}$ where $\varphi\in C_0^{\infty}(\Omega)$. Therefore,
	\begin{align}\label{app_2}
		\begin{split}
			0&=\int_{\Omega}\nabla u\cdot\nabla(\varphi^2\min\{u-1+\delta,0\})dx\\
			&=\int_{\Omega\cap\{u<1-\delta\}}\nabla u\cdot\nabla(\varphi^2\min\{u-1+\delta,0\})dx\\
			&=\int_{\Omega\cap\{u<1-\delta\}}|\nabla u|^2\varphi^2dx+2\int_{\Omega\cap\{u<1-\delta\}}\varphi(u-1+\delta)\nabla u\cdot\nabla\varphi dx,
		\end{split}
	\end{align}
	and so by Caccioppoli like estimate we have
	\begin{align}
		\begin{split}
			\int_{\Omega\cap\{u<1-\delta\}}|\nabla u|^2\varphi^2dx&=-2\int_{\Omega\cap\{u<1-\delta\}}\varphi(u-1+\delta)\nabla u\cdot\nabla\varphi dx\\
			&\leq c\int_{\Omega}u^2|\nabla\varphi|^2dx.
		\end{split}
	\end{align}
	Since $\int_{\Omega}|u|^2dx<\infty$, therefore on passing the limit $\delta\rightarrow 0$ we conclude that $u\in W_{\text{loc}}^{1,2}(\Omega)$. Furthermore, for a nonnegative $\zeta\in C_0^{\infty}(\Omega)$ and $\delta>0$ we have
	\begin{align}\label{app_3}
		\begin{split}
			-\int_{\Omega}&\nabla (u-1)_-\cdot\nabla\zeta dx\\
			=&-\int_{\Omega}\left[\nabla (u-1)_-\cdot\nabla\left(\zeta\max\left\{\min\left\{2-\frac{1-u}{\delta},1\right\},0\right\}\right)\right]dx\\
			\geq&-\int_{\Omega\cap\{1-2\delta<u<1\}}|\nabla (u-1)_-\cdot \nabla\zeta|dx.
		\end{split}
	\end{align}
	Note that by the H\"{o}lder inequality we have
	\begin{align}\label{conv_delta}
		\begin{split}
			0\leq\int_{\Omega\cap\{1-2\delta<u<1\}}|\nabla u\cdot \nabla\zeta|dx\leq&\left(\int_{\Omega\cap\{1-2\delta<u<1\}}|\nabla u|^2dx\right)^{1/2}\\
			&\times\left(\int_{\Omega\cap\{1-2\delta<u<1\}}|\nabla \zeta|^2dx\right)^{1/2}\to 0~\text{as}~\delta\to 0.
		\end{split}
		\end{align}
	Thus $\Delta(u-1)_{-}\geq 0$ in the distributional sense on $\Omega\cap\partial\{u<1\}$. Hence by Theorems $2.3-2.4$ in \cite{kil1} there exists a nonnegative Radon measure $\mu$ (say) such that $\mu=:\Delta(u-1)_{-}$ in $\Omega\cap\partial\{u<1\}$.
	
From \eqref{weak_conv_3}, the positivity of the Radon measure $\mu$ and the usage of Section $9.4$ in {\sc Gilbarg-Trudinger} \cite{Gil_trud} we conclude that $u\in W_{\text{loc}}^{2,q}(\{u\leq 1\}^{\circ})$. Thus $\mu$ is supported on $\Omega\cap\partial\{u<1\}\cap\partial\{u>1\}$ and $u$ satisfies $-\Delta u=0$ in the set $\{u\leq 1\}^{\circ}$.\\
	\noindent In order to prove $(ii)$, we will show that $u_j\rightarrow u$ locally in $C^1(\Omega\setminus\{u=1\})$. Note that we have already proved that $u_j\rightarrow u$ in the $C^1$ norm in a neighbourhood of $\partial\Omega$ of $\bar{\Omega}$. Suppose $M\subset\subset\{u>1\}$. In this set $M$ we have $u\geq 1+2\delta$ for some $\delta>0$. Thus for sufficiently large $j$, with $\delta_j<\delta$, we have $|u_j-u|<\delta$ in $\Omega$ and hence $u_j\geq 1+\delta_j$ in $M$. From \eqref{approx_prob_1} we have $$-\Delta u_j=\lambda (u_j-1)_+^{q-1}~\text{in}~M.$$
	Clearly, $(u_j-1)_+^{q-1}\rightarrow (u-1)_+^{q-1}$ in $L^q(\Omega)$ for $1<q<\infty$ and $u_j\rightarrow u$ uniformly in $\Omega$. This analysis says something more stronger - {\it since $-\Delta u_j=\lambda (u_j-1)_+^{q-1}$ in $M$, we have that $u_j\rightarrow u$ in $W^{2,q}(M)$}. By the embedding $W^{2,q}(M)\hookrightarrow C^1(M)$ for $2<q$, we have $u_j\rightarrow u$ in $C^1(M)$. This shows that $u_j\rightarrow u$ in $C^1(\{u>1\})$. Working on similar lines we can also show that $u_j\rightarrow u$ in $C^1(\{u<1\})$.\\
	\noindent We will now prove $(iii)$. Since $u_j\rightharpoonup u$ in $W_0^{1,2}(\Omega)$, we have that by the weak lower semicontinuity of the norm $\|\cdot\|$ that
	$$\|u\|\leq\lim\inf\|u_j\|.$$
	It is sufficient to prove that $\lim\sup\|u_j\|\leq \|u\|$. To achieve this, we multiply \eqref{approx_prob_1} with $(u_j-1)$ and then integrate by parts. We will also use the fact that $tg\left(\frac{t}{\delta_j}\right)\geq 0$ for any $t\in\mathbb{R}$. This gives,
	\begin{align}\label{weak_conv_4}
	\begin{split}
	\int_{\Omega}|\nabla u_j|^2dx&\leq \lambda\int_{\Omega}f(u_j-1)_+^qdx-\int_{\partial\Omega}\frac{\partial u_j}{\partial\hat{n}}dS\\
	&\rightarrow\lambda\int_{\Omega}(u-1)_+^qdx-\int_{\partial\Omega}\frac{\partial u}{\partial\hat{n}}dS
	\end{split}
	\end{align}
	as $j\rightarrow\infty$. Here $\hat{n}$ is the outward drawn normal to $\partial\Omega$.
\end{proof}

\begin{conjecture}\label{OP_1}
We leave the following as an open conjecture that if one considers the function
$$\Phi(r):=\left(\frac{1}{r^p}\int_{B_r}\frac{\nabla(|\nabla w^+|^{p-2}w^{+^{(p-1)^2}})\cdot\nabla w^+}{|x|^{N-p}}dx\right)\left(\frac{1}{r^p}\int_{B_r}\frac{\nabla(|\nabla w^-|^{p-2}w^{-^{(p-1)^2}})\cdot\nabla w^-}{|x|^{N-p}}dx\right),$$
then the estimate $$\Phi(r)\leq C\left(1+\int_{B_1}\frac{\nabla(|\nabla w^+|^{p-2}w^{+^{(p-1)^2}})\cdot\nabla w^+}{|x|^{N-p}}dx+\int_{B_1}\frac{\nabla(|\nabla w^-|^{p-2}w^{-^{(p-1)^2}})\cdot\nabla w^-}{|x|^{N-p}}dx\right)^2,$$
for each $~0<r\leq 1$ if $\Delta_pw_{\pm}\geq -1$
 may be obtained. However, a big drawback in the case of $p$-Laplace operators is that we do not have $C^2$ regularity of solutions.
 \end{conjecture}

We now establish the free boundary condition for the case of $p>2$ (for the case of $p=2$, the readers may refer to Perera \cite{Perera_NoDEA}). Towards this we choose $\vec{\varphi}\in C_0^1(\Omega,\mathbb{R}^N)$ such that $u\neq 1$ a.e. on the support of $\vec{\varphi}$. On multiplying $\nabla u_n\cdot\vec{\varphi}$ to the weak formulation of \eqref{approx_prob_1} and integrating over the set $\{1-\epsilon^-<u_n<1+\epsilon^+\}$ gives
\begin{align}\label{weak_conv_10}
\begin{split}
\int_{\{1-\epsilon^-<u_n<1+\epsilon^+\}}\left[-\Delta_p u_n+\frac{1}{\alpha_n}g\left(\frac{u_n-1}{\alpha_n}\right)\right]\nabla u_n\cdot\vec{\varphi} dx\\
=\int_{\{1-\epsilon^-<u_n<1+\epsilon^+\}}(u_n-1)_+^{q-1}\nabla u_n\cdot\vec{\varphi} dx.
\end{split}
\end{align}
The term on the left hand side of \eqref{weak_conv_10} can be expressed as follows:
\begin{align}\label{weak_conv_11}
\nabla\cdot\left(\frac{1}{p}|\nabla u_n|^p\vec{\varphi}-(\nabla u_n\cdot\vec{\varphi})|\nabla u_n|^{p-2}\nabla u_n\right)+(\nabla\vec{\varphi}\cdot\nabla u_n)\cdot\nabla u_n|\nabla u_n|^{p-2}&-\frac{1}{p}|\nabla u_n|^p\nabla\cdot\vec{\varphi}\nonumber\\
&+\nabla G\left(\frac{u_n-1}{\alpha_n}\right)\cdot\vec{\varphi}.
\end{align}
Using \eqref{weak_conv_11} and on integrating by parts we obtain
\begin{align}\label{weak_conv_12}
\begin{split}
\int_{\{u_n=1+\epsilon^+\}\cup\{u_n=1-\epsilon^-\}}\left[\frac{1}{p}|\nabla u_n|^p\vec{\varphi}-(\nabla u_n\cdot\vec{\varphi})|\nabla u_n|^{p-2}\nabla u_n+G\left(\frac{u_n-1}{\alpha_j}\right)\hat{\varphi}\right]\cdot\hat{n}dS\\
=\int_{\{1-\epsilon^-<u_n<1+\epsilon^+\}}\left(\frac{1}{p}|\nabla u_n|^p\nabla\cdot\vec{\varphi}-(\nabla\vec{\varphi}\cdot\nabla u_n)|\nabla u_n|^{p-2}\nabla u_n\right)dx\\
+\int_{\{1-\epsilon^-<u_n<1+\epsilon^+\}}\left[G\left(\frac{u_n-1}{\alpha_n}\right)\nabla\cdot\vec{\varphi}+\lambda (u_n-1)_+^{q-1}(\nabla u_n\cdot\vec{\varphi})\right]dx.
\end{split}
\end{align}
The integral on the left of equation \eqref{weak_conv_12} converges to
\begin{align}\label{weak_conv_13}
\begin{split}
&\int_{\{u=1+\epsilon^+\}\cup\{u=1-\epsilon^-\}}\left(\frac{1}{p}|\nabla u|^p\vec{\varphi}-(\nabla u\cdot \vec{\varphi})|\nabla u|^{p-2}\nabla u\right)\cdot\hat{n}dS+\int_{\{u=1+\epsilon^+\}}\vec{\varphi}\cdot\hat{n}dS\\
&=\int_{\{u=1+\epsilon^+\}}\left[1-\left(\frac{p-1}{p}\right)|\nabla u|^p\right]\vec{\varphi}\cdot\hat{n}dS-\int_{\{u=1-\epsilon^-\}}\left(\frac{p-1}{p}\right)|\nabla u|^p\vec{\varphi}\cdot\hat{n}dS.
\end{split}
\end{align}
Thus the equation \eqref{weak_conv_13} under the limit $\epsilon\rightarrow 0$ becomes
\begin{align}\label{weak_conv_14}
0=\underset{\epsilon\rightarrow 0}{\lim}\int_{\{u=1+\epsilon^+\}}\left[\left(\frac{p}{p-1}\right)-|\nabla u|^p\right]\vec{\varphi}\cdot\hat{n}dS-\underset{\epsilon\rightarrow 0}{\lim}\int_{\{u=1-\epsilon^-\}}|\nabla u|^p\vec{\varphi}\cdot\hat{n}dS.
\end{align}
This is because $\hat{n}=\pm\dfrac{\nabla u}{|\nabla u|}$ on the set $\{u=1+\epsilon^{+}\}\cup\{u=1-\epsilon^{-}\}$. This proves that $u$ satisfies the free boundary condition. The solution cannot be trivial as it satisfies the free boundary condition. Thus a solution to \eqref{probmain} exists that obeys the free boundary condition besides the Dirichlet boundary condition.

\begin{remark}\label{existence1}
\noindent Before we prove the existence of a solution to the problem \eqref{probmain}, we sharpen a few tools which that be used in the proof. We observe that
$$I_{\alpha}(u)\leq I(u)~\text{in}~W_0^{1,p}(\Omega).$$ 	
Furthermore, we have
\begin{align}\label{yang-per1}
\begin{split}
I_{\alpha}(u)&\geq\frac{1}{p}\|u\|^p-\frac{\lambda}{q}\int_{\Omega}(u^+)^qdx\\
&\geq \frac{1}{p}\|u\|^p-\frac{C\lambda}{q}\|u\|^q
\end{split}
\end{align}
by Sobolev embedding. Therefore, there exists $\rho_0=\rho_0(\nu,\lambda)>0$ such that
\begin{align}\label{yang-per2}
\begin{split}
I_{\alpha}(u)&\geq\frac{1}{2p}\|u\|^p
\end{split}
\end{align}
for $\|u\|\leq \rho_0$. Furthermore, for a fixed nonzero $u$ we have $I_{\alpha}(tu)\to-\infty$ as $t\to\infty$ and hence there exists a function $v_0$ such that $I_{\alpha}(v_0)<0=I_{\alpha}(0)$. This indicates that the set  $$\Lambda_{\alpha}:=\{\xi\in C([0,1];W_0^{1,p}(\Omega)):\xi(0)=0, I_{\alpha}(\xi(1))<0\}$$ is nonempty. Hence by the Mountain pass theorem we have
\begin{align}\label{yang-per3}
\begin{split}
c_{\alpha}&:=\underset{\xi\in\Lambda_{\alpha}}{\inf}~\underset{u\in\xi([0,1])}{\max}I_{\alpha}(u).
\end{split}
\end{align}
Since by the definition of the set $\Lambda_{\alpha}$ we have $\Lambda\subset\Lambda_{\alpha}$ and
\begin{align}\label{yang-per4}
\begin{split}
c_{\alpha}\leq\underset{u\in\xi([0,1])}{\max}I_{\alpha}(u)\leq\underset{u\in\xi([0,1])}{\max}I(u)
\end{split}
\end{align}
for all $\xi\in\Lambda$. This implies that $c_{\alpha}\leq c$.
\end{remark}

\begin{remark}\label{fin_C}
	Let $\phi_1$ be the first eigenfunction pertaining to the first eigenvalue $\lambda_1$ of $(-\Delta_p)$. We observe that
	\begin{align}\label{ineq1}
	I(t\phi)\to-\infty~\text{as}~t\to\infty.
	\end{align}
	Thus there exists $t_*>0$ such that $I(t_*\phi_1)<0$. Consider the path which is defined by $\xi(t)=t\phi_1$ for $t\in[0,t_*]$. Then $\xi$ yields a path from $\Lambda$ on which
	\begin{align}\label{ineq2}
	I(t\phi_1)\leq\mathcal{D}:=\underset{t\geq 0}{\sup}\int_{\Omega}\left(\frac{\lambda_1}{p}t^p\phi_1+1\right)dx.
	\end{align}
	Therefore $c\leq \mathcal{D}$.
\end{remark}
\begin{proof}[Proof]
	From the Remark \ref{fin_C} we conclude that $c_{\alpha}\leq c\leq \mathcal{D}$. Since $I_{\alpha}$ obeys the (PS) condition, hence a limit of the (PS) sequence, say $u_{\alpha}$, can be proved to be a critical point of $I_{\alpha}$. Thus we have $I_{\alpha}(u_{\alpha})=c_{\alpha}$. Now consider a sequence $\alpha_n$ which converges to zero and name $u_{\alpha_n}$ as $u_n$, $c_{\alpha_n}$ as $c_n$. By the Lemma \ref{convergence1} $(i)-(ii)$ we extract a subsequence of $(u_n)$, still denoted by the same name, converges uniformly in $\bar{\Omega}$, locally in $C^1(\bar{\Omega}\setminus\{u=1\})$ and strongly in $W_0^{1,p}(\Omega)$, to a locally Lipschitz function $u\in W_0^{1,p}(\Omega)\cap C^{1,a}(\bar{\Omega}\setminus H(u))$. Moreover, from \eqref{yang-per2} in Remark \ref{existence1} we have $\lim\sup I_{\alpha_n}(u_n)=\lim\sup c_n\geq \frac{r_0}{4}>0$. This indicates that one of the limit conditions $\lim\sup I_{\alpha_n}(u_n)>0$ or $\lim\inf I_{\alpha_n}(u_n)<0$ in Lemma \ref{convergence1} indeed holds. Hence by the paragraph succeeding Lemma \ref{convergence1} $(iv)$ we conclude that $u$ is nontrivial. Furthermore, by Lemma \ref{convergence1} we have that $u$ is a $C^{2}(\Omega)$ weak solution of $-\Delta_p u=\lambda(u-1)_+^q$ in $\Omega\setminus\partial\{u>1\}$ and the free boundary condition $|\nabla u^+|^p-|\nabla u^-|^p=\frac{p}{p-1}$ in sense of \eqref{weak_conv_14} in addition to vanishing on the boundary $\partial\Omega$.
\end{proof}
\begin{remark}\label{lim_case_reason}
	We note that the Lemma \ref{convergence1} and \ref{conv_res1} are still open questions to be answered for problems that are superlinear in nature and for $p \neq 2$.
\end{remark}
\subsection*{Funding} The research of Jiabin Zuo was supported by the Guangdong Basic and Applied Basic Research Foundation (2022A1515110907) and the Guangdong Basic and Applied Basic Research Foundation (2024A1515012389).
\section*{Acknowledgement}
The author thanks the anonymous referee(s) for his/her constructive comments besides the community of the free boundary value problems for injecting a new lease of life to the study of elliptic PDEs. The author also thanks the NBHM, India (02011/47/2021/NBHM(R.P.)/R\&D II/2615) for the financial support.
\section*{Data availability statement}
Data sharing not applicable to this article as no datasets were generated or analysed during the current study.
\section*{Author Contributions Statement}
All authors wrote, prepared the figures and reviewed the manuscript.



\end{document}